\documentclass[10pt]{amsart}

\usepackage{color}
\usepackage{amssymb, amscd, amsmath, amsthm,  graphicx, latexsym}
\usepackage{multicol}
\usepackage{enumitem}
\usepackage{hyperref}
\usepackage{array}
\usepackage{enumitem}
\usepackage{verbatim}
\newtheorem{theorem}{Theorem}[section]
\newtheorem{proposition}[theorem]{Proposition}
\newtheorem{lemma}[theorem]{Lemma}

\newtheorem{definition}[theorem]{Definition}

\newtheorem{conjecture}[theorem]{Conjecture}
\newtheorem{example}[theorem]{Example}

\title[The Change-Making Problem for 6 coin values and beyond]{The Change-Making Problem\\ for six coin values and beyond}
\author{Cornelia A. Van Cott and Qiyu Zhang}

\begin{document}
\maketitle
\begin{abstract}
The change-making problem asks: given a positive integer $v$ and a collection $\mathcal{C}$ of integer coin values $c_1=1<c_2< c_3< \cdots< c_n$, what is the minimum number of coins needed to represent $v$ with coin values from $\mathcal{C}$? For some coin systems, the greedy algorithm finds a representation with a minimum number of coins for all $v$. Such coin systems are called orderly. Characterizations of orderly coin systems with 3, 4, 5, and 6 coin values, respectively, have been found. We give a simplified characterization of orderly coin systems with 6 coin values. Our technique of proof extends to larger coin systems, and we find orderly coin systems with an arbitrarily large number of coins $(c_1, c_2, \ldots, c_n)$ such that $(c_1, c_2, \ldots, c_i)$ is not orderly for all $3 < i < n$. 
\end{abstract}

\section{Introduction}

As customers, we trust store cashiers to give us our change using the minimum possible number of coins. Most of the time, cashiers get this done in a heartbeat. The secret to their success is that they use the so-called {\em greedy algorithm} to make change. Given a positive integer $v$, this algorithm works as follows. Find the largest coin value less than or equal to $v$. Take as many coins of that value as possible without the sum of the coins' values exceeding $v$. Then move to the next largest coin denomination and repeat the process until the entire collection's value sums up to $v$. For both the United States' coin system $(1,5,10,25)$ and the EU coin system $(1, 2, 5, 10, 20, 50, 100, 200)$, the greedy algorithm always produces the minimum number of coins. 

Since the greedy algorithm serves us so well with these familiar coin systems, it often comes as a surprise that the greedy algorithm fails to be optimal in other settings. Suppose, for example, that the United States replaced the dime and quarter with 15 and 20 cent coins, making the coin system $(1,5,15,20)$. If we needed 30 cents change, the greedy algorithm will produce a solution with three coins: 20 + 5 + 5. But we can get the job done with just two coins: 15 + 15. 

This leads us to the so-called {\em change-making problem}. Given a positive integer $v$ and a coin system $\mathcal{C} = (c_1, c_2, c_3, \ldots, c_n)$ with $c_1=1<c_2< c_3< \cdots< c_n$, what is the minimum number of coins needed to represent $v$ with coins from $\mathcal{C}$? Denote this minimum number via $opt_\mathcal{C}(v)$. We assume that there are an unlimited number of each type of coin available. The fact that the smallest coin value is always $c_1=1$ ensures that each positive integer $v$ has at least one representation. 

A collection of coins from coin system $\mathcal{C}$ with total value $v$ is called an {\em optimal solution} if the collection has a minimum number of coins. An optimal solution need not be unique. For instance, suppose the coin system is $\mathcal{C}=(1,3,4)$. Then there are two optimal solutions for $v = 9$ as follows: $3+3+3$ and $4+4+1$. 

For some coin systems, the greedy algorithm solves the change-making problem for all $v>0$. Let $grd_\mathcal{C}(v)$ denote the number of coins in the greedy representation of $v$ with coin system $\mathcal{C}$. We make the following definition.

\begin{definition}
A coin system $\mathcal{C}=(1, c_2, c_3, \ldots, c_n)$ is {\bf orderly}\footnote{In addition to this term, other terms used in place of {\em orderly} in the literature are canonical~\cite{cai, Kozen, Miyashiro}, greedy~\cite{Cowen}, and standard~\cite{Tien}.} if for all $v>0$,  $grd_\mathcal{\mathcal{C}}(v) = opt_\mathcal{\mathcal{C}}(v)$.
\end{definition}

The United States' coin system and the EU coin system are both orderly (we will justify this later), while the coin system $(1,5,15,20)$ is not orderly since the greedy algorithm failed to be optimal for $v=30$. For a non-orderly coin system $\mathcal{C}$, any positive integer $v$ such that $grd_\mathcal{C}(v) \neq opt_\mathcal{C}(v)$ is called a {\bf counterexample.}

The change-making problem is a special case of the knapsack problem. In this more general setting, suppose you have $n$  types of objects available to pack into a knapsack, and suppose the monetary value of each object is denoted by $c_1, c_2, \ldots, c_n$, respectively. These objects have weights denoted by $w_1, w_2, \ldots, w_n$, respectively. The knapsack problem asks how we can pack the knapsack in such a way that the monetary value of the content equals some value $v$ and the total weight of the packed knapsack is {\em minimized}. The change-making problem, then, is the special case in which the items have weight $w_i=1$ for all $i$. For a discussion of the knapsack problem, see survey articles~\cite{Cacchiani1, CACCHIANI2} and books~\cite{kellerer, martello}. 

Over the last several decades, progress has been made on the change-making problem. See Section~\ref{pastresults} for a summary of the big results. In particular, characterizations of orderly coin systems with 3, 4, and 5 coin values are reviewed in Theorems~\ref{three},~\ref{four}, and~\ref{five}, respectively. In 2023, a characterization of orderly coin systems with six coins was given by Miyashiro and Suzuki as follows.

\begin{theorem}\cite{Miyashiro}\label{six}
The coin system $\mathcal{C} = (1, c_2, c_3, c_4, c_5, c_6)$ is orderly if and only if one of the following holds:
\begin{enumerate}
    \item $(1,c_2,c_3,c_4,c_5)$ is orderly and $grd_{\mathcal{C}}(mc_5) \leq m$ where $m = \lceil \frac{c_6}{c_5}\rceil$.
    \item $(1,c_2,c_3,c_4,c_5)$ is not orderly and $\mathcal{C}$ is one of (a), (b), or (c), where\\ $\ell = \lceil \frac{c_5}{c_3} \rceil$. In addition, 
    $grd_{\mathcal{C}}(\ell c_3) =  \ell c_3 - c_5 + 1 -\lfloor (\ell c_3 - c_5)/c_2 \rfloor(c_2-1)$.
    \begin{enumerate}
    \item $\mathcal{C} = (1,2,3,a, a+1, 2a)$ and $a \geq 5$;
    \item $\mathcal{C} = (1,a,2a-1,b,a+b-1,2b-1),$ $b\geq 3a-1$, and $grd_{\mathcal{C}}(\ell c_3) \leq \ell$;
    \item $\mathcal{C} = (1,a,2a,b,a+b,2b),$ $b\geq 3a-1$, $b \neq 3a$, and $grd_{\mathcal{C}}(\ell c_3) \leq \ell$.
    \end{enumerate}
\end{enumerate}
\end{theorem}

While orderly coin systems of size 3, 4, and 5 provide a foundation for the subject, they yield limited intuition about larger coin systems. Thus the result for six coin values is critical in the development of more general results.

We independently found a characterization of orderly coin systems with six coins which simplifies the above characterization. While the characterization of Miyashiro and Suzuki splits into two cases depending on whether or not the coin system $(1, c_2, c_3, c_4, c_5)$ is orderly, we split into four cases as follows. 

 Given any coin system $\mathcal{C} = (1, c_2, c_3, \ldots, c_n)$, a {\bf prefix coin system} of $\mathcal{C}$ is defined as $\mathcal{C}' = (1, c_2, c_3, \ldots, c_i)$ where $1\leq i \leq n$. We can associate to any $n$-value coin system a string of $n$ symbols where the $i$th symbol is $+$ if the length $i$ prefix coin system is orderly and is $-$ otherwise. So for instance, the coin system $(1,2,5,6,10)$ has pattern $+ + + - +$, because all the prefix coin systems are orderly except for the fourth $(1,2,5,6)$ which has a counterexample of 10. Every orderly coin system with 4 or more coin values has a pattern that begins with $+++$ (Theorem~\ref{start+}). Therefore, orderly coin systems with 6 coin values have one of four possible $+/-$ patterns. Our result is as follows.

\begin{theorem}\label{main}
A coin system $\mathcal{C} = (1, c_2, c_3, c_4, c_5, c_6)$ is  orderly if and only if $\mathcal{C}$ is one of the following:
\begin{enumerate}
    \item Coin systems with pattern $+ + + + - +$.
    \begin{enumerate}
    \item $(1, 2, 3, a, a+1, 2a)$, where $a\geq 5$
    \item $(1,a,2a,b,b+a, 2b)$, 
    \\where $a\geq 2$, $b\geq 3a-1$, $b \neq 3a$, and  $grd_\mathcal{C}(2ma) \leq m $ where $m = \lceil \frac{b}{2a} \rceil$
    \item $(1,a,2a-1, b, b + a - 1, 2b - 1),$ 
    \\where $a\geq 2$, $b\geq 3a-1$, and $grd_\mathcal{C}(m(2a-1)) \leq m $ where $m = \lceil \frac{b}{2a-1} \rceil$
    \end{enumerate}
    \item Coin systems with pattern $ + + + - - +$.
 \begin{enumerate}
     \item $(1,a,2a-1,m(2a-1)-(a-1),m(2a-1),(2m-1)(2a-1))$, \\where $1<m<a$.
\item $(1,a,2a,m(2a-1)-(a-1),m(2a-1)+1, (2m-1)(2a-1)+1)$, 
\\where $1<m\leq a$.
 \end{enumerate}
    \item Coin systems with pattern $+ + + + + +$ or $+ + + - + +$.\\
$(1, c_2, c_3, c_4, c_5)$ is orderly and $grd_\mathcal{C}(mc_5) \leq m $ where $m = \lceil \frac{c_6}{c_5} \rceil$.
\end{enumerate}
\end{theorem}

Proving that Theorems~\ref{six} and~\ref{main} give equivalent characterizations is not particularly difficult, but it is long. We omit it here. 

Comparing the characterizations in Theorems~\ref{six} and~\ref{main}, notice that Theorem~\ref{main} entirely removes the lengthy condition $$grd_C(\ell c_3) = \ell c_3 - c_5 + 1 -\lfloor (\ell c_3 - c_5)/c_2 \rfloor(c_2-1),$$
where $\ell = \lceil c_5/c_3 \rceil$. In particular, Theorem~\ref{main}(2) has no special conditions.

Another key difference between Theorem~\ref{six} and Theorem~\ref{main} is the method of proof. The main difficulty for both proofs is showing that the given conditions on $c_i$ are sufficient for $(1, c_2, c_3, c_4, c_5, c_6)$ to be orderly. The strategy used by Miyashiro and Suzuki is to work carefully through several possibilities. Given any value $v>0$, they consider the cases (1) $v < c_4$, (2) $c_4 < v < c_5$, (3) $c_5 < v < c_6$, and (4) $c_6 < v$. For each case, they show that $v$ is not a counterexample, and thus conclude that the coin system is orderly. 

In our proof, we use a result of Pearson (Theorem~\ref{Pearson}) which roughly says that a minimal counterexample, if it exists, is related to the coin values of $\mathcal{C}$ in a particular way. The strength of this approach is that it can be used on larger coin systems, as well. In comparison, checking for counterexamples between consecutive coin values becomes burdensome as the number of coin values increases.

In Section~\ref{extensions}, we exhibit that we can generalize our results and our method to find orderly coin systems with $n$ values. We consider orderly coin systems such that $(c_1, c_2, \ldots, c_i)$ is not orderly for all $3 < i < n$, which we denote by $+++-\cdots -+$. Many coin systems we observed with this pattern have a particular structure; we call them {\bf fixed gap coin systems} (Definition~\ref{fixedgap}). We give  three infinite families of fixed gap coin systems $\mathcal{D}$, $\mathcal{E}$, and $\mathcal{F}$ (see Theorems~\ref{D},~\ref{E}, and~\ref{F}). Coin systems in $\mathcal{D}$ have $3r+2$ coin types where $r\geq 1$, while $\mathcal{E}$ and $\mathcal{F}$ both have $3r$ coin types where $r\geq 2$. 

In Section~\ref{future}, we discuss the general case of finding all coin systems with pattern $+++-\cdots -+$. Conjecture~\ref{conjecture} summarizes our observations. \\

\noindent{\bf Acknowledgements.} We thank the referees for their helpful comments and suggestions.

\section{Background}~\label{pastresults}

A priori, checking whether an arbitrary coin system $\mathcal{C}$ is orderly is an infinite task, requiring us to confirm that $grd_\mathcal{C}(v) = opt_\mathcal{C}(v)$ for every $v>0$. The following result narrows down the location of the smallest counterexample of a coin system, if it exists.
\begin{theorem}\cite{Kozen}\label{subset}
If the coin system $\mathcal{C} = (1, c_2, c_3, \ldots, c_n)$ is not orderly, then the smallest counterexample $v$ lies in the range: $c_3 < v < c_{n-1}+c_n. $
\end{theorem}
The next result improves on the previous one and enables us to determine whether a coin system is orderly in $O(n^3)$ time. We will make use of this theorem throughout our work here. First we set up some notation and terminology.  

Say we have a coin system $\mathcal{C} = (1, c_2, c_3, \ldots, c_n)$, and suppose we found a way to represent the value $v$ via coins as follows:  $v = x_1 +x_2c_2 \cdots + x_nc_n$ where each $x_i$ is a nonnegative integer. We express this representation of $v$ by coins in $\mathcal{C}$ via a vector as follows: $(x_1, x_2, \ldots, x_n)_\mathcal{C}.$ Such a vector representing $v$ need not be unique, and we can order all such representations as follows.

\begin{definition}
Let $v$ be a positive integer and suppose that ${\vec x} = (x_1, x_2, \ldots, x_n)_\mathcal{C}$ and ${\vec x}' = (x_1', x_2', \ldots, x_n')_\mathcal{C}$ are two representations of the value $v$ with respect to the coin system $\mathcal{C} = (1, c_2, c_3, \ldots, c_n)$. The representation ${\vec x}$ is {\bf lexicographically smaller} than ${\vec x}'$ if $x_i = x_i'$ for $1\leq i\leq k < n$ for some $k$ and $x_{k+1} < x_{k+1}'$. 
\end{definition}

\begin{theorem}\cite{Pearson}\label{Pearson}
Let $\mathcal{C} = (1, c_2, c_3, \ldots, c_n)$ be a non-orderly coin system. Let $w$ be the minimum counterexample for $\mathcal{C}$. Suppose that the lexicographically smallest optimal solution for $w$ with respect to $\mathcal{C}$ is 
$$(0,0,\ldots, 0,x_i,x_{i+1},\ldots,x_j,0,0,\ldots,0)_\mathcal{C} $$
where $1\leq i\leq j <n$ and $x_i, x_j>0$.  Then the greedy representation for $c_{j+1}-1$ in $\mathcal{C}$ is
$$(y_1, y_2, \ldots, y_{i-1},x_i - 1, x_{i+1}, \ldots, x_j,0, 0, \ldots,0)_\mathcal{C} $$
where each $y_i$ is a nonnegative integer. 
\end{theorem}

Now we discuss results for prefix coin systems (defined in the Introduction). We have a term for orderly coin systems where all prefix coin systems are orderly.
\begin{definition}
A coin system $\mathcal{C}=(1, c_2, c_3, \ldots, c_n)$ is {\bf totally orderly} if each coin system $(1, c_2, c_3, \ldots, c_i)$ is orderly for all $i=1, 2, \ldots, n$.
\end{definition}

Sometimes knowing information about a prefix coin system yields information about the coin system as a whole. 

\begin{theorem}\cite{Adamaszek}\label{start+}
If $\mathcal{C} = (1,c_2, c_3, \ldots, c_n)$ is orderly, then the coin system $(1,c_2, c_k)$ is orderly for all $3\leq k \leq n$. In particular, $(1, c_2, c_3)$ is orderly.
\end{theorem}
For example, consider the coin system $(1,3,4)$. The value 6 is a counterexample for this coin system (the greedy solution is $4+1+1$ while an optimal solution is $3+3$). Thus, the pattern for this coin system is $++-$. One might think that adding new coin values larger than 4 to the coin system could ``fix" it and make it orderly. But Theorem~\ref{start+} tells us that this is impossible. There is no collection of higher coin values which one could add to the coin system $(1,3,4)$ which would cause the resulting coin system to be orderly.  A more involved example involving 7-coins is as follows.

\begin{theorem}\cite{Adamaszek}\label{impossible}
There are no coin systems with the pattern $+++-+-+$.
\end{theorem}
The theorem above implies that no orderly currency has a prefix coin system with pattern $+++-+-+$. 

The following theorem (sometimes called the {\em One Point Theorem}) gives a necessary and sufficient condition to determine whether a coin system $(1, c_2, \ldots, c_{n-1}, c_n)$ is orderly, if one already knows that the prefix coin system $(1, c_2, \ldots, c_{n-1})$ is orderly.

\begin{theorem}[One Point Theorem]\cite{Adamaszek, Hu, Magazine}\label{OPT}
Suppose that $\mathcal{C}' = (1, c_2, \ldots, c_{n-1})$ is orderly and $c_{n-1} < c_{n}$. Let $m = \lceil \frac{c_n}{c_{n-1}} \rceil$. The following are equivalent:
\begin{enumerate}
    \item The coin system $\mathcal{C} = (1, c_2, \ldots, c_{n-1}, c_n)$ is orderly. 
    \item $grd_\mathcal{C}(mc_{n-1}) = opt_\mathcal{C}(mc_{n-1})$
   \item $grd_\mathcal{C}(mc_{n-1}) \leq m$
\end{enumerate}
\end{theorem}

The known results about orderly coin systems with only a few coin values are as follows. Any coin system with two coin values $\mathcal{C} = (1, c)$ where $c >1$ is orderly.  The orderly 3-value coin systems are then characterized by the One Point Theorem (Theorem~\ref{OPT}). An equivalent characterization, together with characterizations of orderly coin systems with 4 or 5 coins, are as follows.

\begin{theorem}\cite{Adamaszek, Kozen}\label{three}
The coin system $\mathcal{C} = (1, c_2, c_3)$ is orderly if and only if $c_3 - c_2\in \mathcal{A}$, where
$\mathcal{A}$ is defined as follows 
$$\mathcal{A} = \{c_2-1, c_2\}\cup\{2c_2-2, 2c_2-1, 2c_2\}\cup \cdots \cup \{mc_2-m,\ldots, mc_2\}\cup\cdots$$
\end{theorem}

\begin{theorem}\cite{Adamaszek, Cowen}\label{four}
The coin system $\mathcal{C} = (1, c_2, c_3, c_4)$ is orderly if and only if it is totally orderly.
\end{theorem}

\begin{theorem}\cite{Adamaszek, Cowen}\label{five}
The coin system $\mathcal{C} = (1, c_2, c_3, c_4, c_5)$ is orderly if and only if one of the following holds:
\begin{enumerate}
    \item $(1,c_2,c_3,c_4,c_5)=(1,2,a,a+1,2a)$ for some $a\geq 4$, in which case $(1,c_2,c_3,c_4)$ is not orderly,
    \item $\mathcal{C}$ is totally orderly.
\end{enumerate}
\end{theorem}

Notice that Theorem~\ref{three} and then repeated applications of Theorem~\ref{OPT} imply that the United States' coin system $(1,5,10,25)$ and the EU coin system $(1, 2, 5, 10, 20, 50, 100, 200)$ are both totally orderly. Characterizations for orderly coin systems with 6 coins were given in the introduction in Theorems~\ref{six} and~\ref{main}.

Next, we state constraints on the differences between consecutive coin values for orderly coin systems. The first is simple.

\begin{theorem}\cite{Adamaszek}\label{firstjump}
If $\mathcal{C} = (1, c_2, c_3, \ldots, c_n)$ is orderly, then for all $i = 2,3 \ldots, n$,
$$c_i - c_{i-1} \geq c_2 - 1.$$
\end{theorem}

The next result is more involved. Roughly speaking, if two consecutive differences between adjacent coin values are big, then {\em all} subsequent differences between adjacent coin values are also big.
\begin{theorem}\cite{Adamaszek}\label{jumps}
Suppose $\mathcal{C}=(1, c_2, c_3, \ldots, c_n)$ is orderly. If $m\geq 3$ and 
$$c_{m-1} > 2c_{m-2}~~\text{ and }~~c_m>2c_{m-1}, $$
then for every $t\geq m$, we have $c_{t+1} - c_t \geq c_m - c_{m-1}$.
\end{theorem}

To finish this section, we discuss properties of non-orderly coin systems and their counterexamples. Consider again the coin system $\mathcal{C} = (1, 3, 4)$. The smallest counterexample for this coin system was $w=6$. The greedy solution is $4+1+1$, while an optimal solution is $3+3$. Observe that the set of coin denominations used in creating the greedy solution $\{1,4\}$ was disjoint from the coin denominations used to create the optimal solution $\{3\}$. This is true in general, as the result below states.

\begin{proposition}\cite{Chang}\label{Chang}
Suppose a coin system $\mathcal{C} = (1, c_2, c_3, \ldots, c_n)$ is not orderly. Let $w$ be the smallest counterexample, and suppose that the greedy solution for $w$ is given by 
$w = x_1\cdot 1 + x_2\cdot c_2 + x_3\cdot c_3 + \cdots+ x_n\cdot c_n.$
Suppose an optimal solution for $w$ is given by:
$w = y_1\cdot 1 + y_2\cdot c_2 + y_3\cdot c_3 + \cdots + y_n\cdot c_n.$
Then $x_iy_i = 0$ for all $i$. 
\end{proposition}

The final result considers the situation where a non-orderly coin system has another prefix coin system that is also non-orderly. This result brings up another concept (see, for example,~\cite{cai}).

\begin{definition}
 A coin system $\mathcal{C} = (1, c_2, c_3, \ldots, c_n)$ is {\bf tight} if it has no counterexample smaller than $c_n$.
\end{definition}

\begin{theorem}\cite{cai}\label{two-not-orderly}
Let $\mathcal{C}_1 = (1,c_2,c_3)$, $\mathcal{C}_2 = (1,c_2,c_3,\ldots ,c_m)$ and\\ $\mathcal{C}_3 = (1, c_2, c_3,\ldots, c_m, c_{m+1})$ be three tight coin systems such that $\mathcal{C}_1$ is orderly but $\mathcal{C}_2$ is not. If $\mathcal{C}_3$ is not orderly, then there is a counterexample $x = c_i + c_j > c_{m+1}$ of $\mathcal{C}_3$ with $1<c_i \leq c_j \leq c_m$.
\end{theorem}

There are several perspectives and extensions of the change-making problem which we leave untouched, since they do not pertain to our goals. Interested readers can check them out~\cite{Tien, chan, ChangGill, Goebbels, Rybin}. We note that~\cite{jones} contains apparently elegant necessary and sufficient conditions for a coin system to be orderly, but the result was later shown to be false~\cite{maurer}.

\section{Orderly coin systems with six coin values}\label{mainresult}
We are ready to prove Theorem~\ref{main}, which gives a characterization of orderly coin systems with 6 coin values. In light of Theorem~\ref{start+}, each such coin system has one of the following four patterns:
\begin{enumerate}
\begin{multicols}{2}
    \item $ + + + + - +$
    \item $+ + + - - +$
    \item $+ + + + + +$
    \item $+ + + - + + $
\end{multicols}
\end{enumerate}
Part (1) of the theorem considers coin systems with pattern $ + + + + - +$ and is proved in Section~\ref{section+-+}. Part (2) of the theorem considers the pattern $+ + + - - +$ and is proved in Section~\ref{--+}. Part (3), which considers the remaining two patterns, directly follows from the One Point Theorem (Theorem~\ref{OPT}). 

\subsection{Coin systems with pattern: $++++-+$}\label{section+-+}\hfill

In this section, we prove part (1) of Theorem~\ref{main}. We restate the result here for clarity.

\begin{theorem}\label{+-+} A coin system has pattern $++++-+$ if and only if it is one of the following:
\begin{enumerate}[label=\alph*.]
    \item $(1, 2, 3, a, a+1, 2a)$, where $a\geq 5$.

    \item $(1,a,2a,b,b+a, 2b)$, \\where $a\geq 2$, $b\geq 3a-1$, $b \neq 3a$, and  $grd(2ma) \leq m $ where $m = \lceil \frac{b}{2a} \rceil$.
    
    \item $(1,a,2a-1, b, b + a - 1, 2b - 1),$ \\where $a\geq 2$, $b\geq 3a-1$, and $grd(m(2a-1)) \leq m $ where $m = \lceil \frac{b}{2a-1} \rceil$.

\end{enumerate}
\end{theorem}

\begin{proof}
Suppose a coin system has the form $++++-+$. Since $(1, c_2, c_3, c_4)$ is orderly and $(1, c_2, c_3, c_4, c_5)$ is {\em not} orderly, we know by the One Point Theorem that $mc_4$ is a counterexample for $(1, c_2, c_3, c_4, c_5)$ where $m = \lceil \frac{c_5}{c_4} \rceil$. Moreover, it must be the case that $c_6\leq  mc_4$, for otherwise $\mathcal{C}$ would not be orderly. 

Now observe that
$$c_6 \leq mc_4 = \lceil \tfrac{c_5}{c_4} \rceil c_4 < (\tfrac{c_5}{c_4}+1)c_4 = c_5 + c_4 < c_5 + c_5.$$
Therefore 
\begin{align*}
c_6 &< c_5 + c_5, \\
c_6 &< c_5 + c_4.
\end{align*}
Now since $(1, c_2, c_3, c_4, c_5, c_6)$ is orderly, there must exist coins $c_i$ and $c_j$ where $1\leq i<j \leq 4$ such that 
\begin{align}
   c_6 + c_j &= c_5 + c_5\label{1}, \\
      c_6 + c_i &= c_5 + c_4 \label{2}.
\end{align}
Subtracting these two equations, we have
\begin{align*}
    c_5-c_4 &= c_j-c_i < c_j\leq c_4
\end{align*}
Therefore, $c_5 \leq 2c_4$. We conclude that $1< \frac{c_5}{c_4} \leq 2$, which implies that $m=\lceil \tfrac{c_5}{c_4} \rceil=2$.

We now have three main cases to consider for the possible values of $j$ (and within each case, we will consider the values $i<j$): Case (1) $j=4$, Case (2) $j=3$, and Case (3) $j=2$.

{\bf Case 1: $\boldsymbol{j = 4}$.} Plugging in and combining Equations~\ref{1} and~\ref{2} above, we have 
$c_5 + c_i = 2c_4$. This shows that $2c_4$ is not a counterexample of $(1, c_2, c_3, c_4, c_5)$, which is a contradiction. Therefore, $j=4$ is impossible.

{\bf Case 2a: $\boldsymbol{j=3}$ and $\boldsymbol{i=1}$.} 
Equations~\ref{1} and~\ref{2} above become:
\begin{align*}
    c_6+c_3 &=c_5 +c_5,\\
    c_6 + 1 &= c_5 + c_4.
\end{align*}
Observe that $c_5 + c_4 = c_6 + 1 >2c_4 \geq c_6$ which implies that $2c_4 = c_6$. Setting $c_4 = a$, it follows that $c_6 = 2a$, $c_5 = a+1$, and $c_3 = 2$. However, since $1<c_2<c_3$, we are left without any possible value for $c_2$, which is a contradiction. Therefore, it is impossible to have $j=3$ and $i=1$.

{\bf Case 2b: $\boldsymbol{j=3}$ and $\boldsymbol{i=2}$.} 
Equations~\ref{1} and~\ref{2} become:
\begin{align*}
    c_6+c_3 &=c_5 +c_5,\\
    c_6 +c_2 &= c_5 + c_4.
\end{align*}
Subtracting these two equations. we have
\begin{align*}
    c_5-c_4 &= c_3-c_2 = d
\end{align*}
where we define $d$ is the gap between the coins. Simplifying the equations, we conclude that
\begin{itemize}
\begin{multicols}{3}
 \item[~]   $c_3 = c_2+d$
   \item[~]  $c_5 = c_4+d$
   \item[~]  $c_6 = 2c_4-c_2+d$
    \end{multicols}
\end{itemize}

So then, the coin system is as follows, 
$$(1, c_2, c_3, c_4, c_5, c_6) = (1,c_2,c_2+d,c_4,c_4+d, 2c_4-c_2+d)$$
By Theorem~\ref{firstjump}, we know $d\geq c_2 - 1$. On the other hand, because $c_6\leq 2c_4$, it follows that $d\leq c_2$. Therefore, there are only two possibilities: $d = c_2$ or $d = c_2-1$. So the coin system has one of two possible forms:
$$\mathcal{C} = (1,c_2,2c_2,c_4,c_4+c_2, 2c_4) ~~~\text{ or }~~~ \mathcal{C}' =(1,c_2,2c_2-1, c_4, c_4 + c_2 - 1, 2c_4 - 1).$$

There are further constraints on $c_4$. First, Theorem~\ref{firstjump} tells us that the skip between the third and fourth coins is at least $c_2-1$. Therefore, the coin system $\mathcal{C}$  satisfies $c_4\geq 3c_2-1$, and the coin system $\mathcal{C}'$ satisfies $c_4\geq 3c_2-2$. 

Because $2c_4$ is a counterexample for the 5-value coin system, the difference $2c_4 - c_5$ is not a coin value in the coin system. Therefore, for $\mathcal{C}$ we know $c_4 \neq 3c_2$, and $\mathcal{C}'$ satisfies $c_4\neq 3c_2-2$.

Finally, since the 4-value coin system is orderly, we know by the One Point Theorem that the coin system $\mathcal{C}$ satisfies  $grd_\mathcal{C}(2mc_2) \leq m $ where $m = \lceil \frac{c_4}{2c_2} \rceil$, and the coin system $\mathcal{C}'$ satisfies $grd_{\mathcal{C}'}(m(2c_2-1)) \leq m$ where $m= \lceil \frac{c_4}{2c_2-1} \rceil$. 

Setting $c_2=a$ and $c_4=b$, the coin systems $\mathcal{C}$ and $\mathcal{C}'$ coincide with parts (b) and (c) in the statement of the Theorem, respectively.

{\bf Case 3: $\boldsymbol{j = 2}$.} In this case, we must have $i=1$. Equations~\ref{1} and~\ref{2} become
\begin{align*}
    c_6+c_2 &=c_5 +c_5,\\
    c_6 + 1 &= c_5 + c_4.
\end{align*}
Now observe that $c_5 + c_4 = c_6+1 > 2c_4 \geq c_6$. This forces $2c_4 = c_6$. Setting $c_4 = a$, and simplifying all these equations, we conclude that
\begin{itemize}
\begin{multicols}{3}
 \item[~]   $c_2 = 2$
   \item[~]  $c_5 = a+1$
   \item[~]  $c_6 = 2a$
    \end{multicols}
\end{itemize}
So then, gathering everything together, we conclude that 
$$(1, c_2, c_3, c_4, c_5, c_6) = (1,2,b,a,a+1, 2a)$$
for some integers $a$ and $b$ with $2<b<a$.

We now use Theorem~\ref{jumps} to show that the values of $b$ and $a$ cannot both be large. Suppose that $b>4$ and $a>2b$. Then Theorem~\ref{jumps} implies that 
$$(a+1) - a \geq a-b$$
which simplifies to $a\leq b+1$, but this contradicts $a>2b$.
Therefore, we cannot have both $b>4$ and $a>2b$, so it must be the case that $b\leq 4$ or $a \leq 2b$.

Suppose that $b > 4$ and $b < a \leq 2b$. We can rule out this possibility by observing that the resulting coin system $(1, 2, b, a, a+1, 2a)$ will have $c=a+b$ as a counterexample, so the coin system is not orderly.

Finally, suppose that $b\leq 4$, which implies that $b=3$ or $b=4$. We can immediately rule out the possibility that $b=4$, since in that case our coin system is $(1, 2, 4, a, a+1, 2a)$ and the value $c = a+4$ is a counterexample, making our coin system non-orderly. If $b=3$, we obtain the coin system $(1, 2, 3, a, a+1, 2a)$. If $a=4$ the system is totally orderly, so we must have $a\geq 5$, as desired. 

Conversely, suppose that we have a coin system that is in one of the three forms stated in the theorem. We check that the coin systems have the pattern $++++-+$. 

We first show the coin system $(1, 2, 3, a, a+1, 2a)$ where $a\geq 5$ has pattern $++++-+$. The coin system $(1,2,3)$ is orderly by Theorem~\ref{three}. The fact that $(1,2,3,a)$ is orderly for $a\geq 5$ follows from Theorem~\ref{OPT} and the fact that $(1,2,3)$ is orderly. The 5-value prefix coin system $(1, 2, 3, a, a+1)$ for $a\geq 5$ is not orderly because $2a$ is a counterexample (this is where we need the assumption that $a\geq 5$).

All that remains is to show that the 6-value coin system is orderly. Theorem~\ref{Pearson} enables us to find the minimum counterexample (if it exists) as follows. First, set  $\mathcal{C}=(1,c_2,c_3,c_4,c_5,c_6)=(1, 2, 3, a, a+1, 2a)$. We find the greedy representation of $c_{k}-1$ for $1< k \leq 6$ and express it as a vector with respect to $\mathcal{C}$. By Theorem~\ref{Pearson}, the minimum counterexample (if it exists) is found as follows. For each vector corresponding to $c_k-1$ and for each $p$ where $1\leq p < k-1$, set the first $p$ entries of the vector to 0, and add 1 to the $p+1^{th}$ entry. The resulting values are the candidates for minimal counterexample. 

Doing this process to the vector associated to $c_2-1$ is not possible because we would have to have $1\leq p< 1$ (indeed, it is always true that the case $c_2 - 1$ may be disregarded). Doing this process for $c_3-1= (0,1,0,0,0,0)_\mathcal{C}$ with $p=1$ yields the vector 
$(0,2,0,0,0,0)_\mathcal{C}$, which corresponds to the value 4. We know 4 is not a counterexample, since the greedy solution for $4$ uses only 2 coins.

Next we consider $c_4-1=a-1$. The greedy solution for this value can be expressed in vector form as $(x_1,x_2,x_3,0,0,0)_\mathcal{C}$ where $x_1+2x_2+3x_3=a-1$ and $x_1, x_2 \in\{0,1\}$ and $x_1x_2=0.$ With $p=1$, we have $(0,x_2+1,x_3,0,0,0)_\mathcal{C},$ which corresponds to the value:
\begin{align*}
2(x_2+1)+3x_3 &=\begin{cases}
    a &\text{ if } x_1=1 \\
    a+1 &\text{ if } x_1=0 
\end{cases}
\end{align*}
For $p=2$, we have $(0,0,x_3+1,0,0,0)_\mathcal{C},$ which corresponds to the value:
\begin{align*}
3(x_3+1) =\begin{cases}
    a&\text{ if } x_1=0, x_2=1\\
    a+1 &\text{ if } x_1=1, x_2=0\\
    a+2&\text{ if } x_1=0, x_2=0\\
\end{cases}
\end{align*}
None of these values $a$, $a+1$, or $a+2$ are counterexamples, because each has a greedy solution with 1 or 2 coins. 

Next up, the value $c_5-1 = a$ has greedy representation $(0,0,0,1,0,0)_\mathcal{C}$ and yields three candidates for counterexample (corresponding to $p=1,2,3$):
\begin{align*}
(0,1,0,1,0,0)_\mathcal{C}&=a+2\\
(0,0,1,1,0,0)_\mathcal{C}&=a+3\\
(0,0,0,2,0,0)_\mathcal{C}&=2a
\end{align*}
Again, the values $a+2$, $a+3$, and $2a$ are not counterexamples, since each has a greedy solution with 1 or 2 coins. 

Finally, we consider the vector associated to $c_6-1=2a-1$. The greedy solution has vector representation $(y_1,y_2, y_3,0,1,0)_\mathcal{C}$ where $y_1 +2y_2+3y_3=a-2$ and  $y_1, y_2 \in\{0,1\}$ and  $y_1y_2=0.$ With $p=1$, we have the vector $(0,y_2+1,y_3,0,1,0)_\mathcal{C},$ which corresponds to the value: 
\begin{align*}
2(y_2+1)+3y_3+a+1=\begin{cases}
    2a &\text{ if } y_1=1\\
    2a+1 &\text{ if } y_1=0\\
\end{cases}
\end{align*}
For $p=2$, we have the vector $(0,0,y_3+1,0,1,0)_\mathcal{C},$ which corresponds to the value:
\begin{align*}
3(y_3+1) + a+1=
\begin{cases}
    2a &\text{ if } y_1=0, y_2=1\\
    2a+1 &\text{ if } y_1=1, y_2=0\\
    2a+2 &\text{ if } y_1=0, y_2=0\\
\end{cases}
\end{align*}
For $p=3$ and $p=4$, we have:
\begin{align*}
(0,0,0,1,1,0)_\mathcal{C}&=2a+1\\
(0,0,0,0,2,0)_\mathcal{C}&=2a+2
\end{align*}
Altogether from these four cases, we generate three possible candidates for counterexample: $2a$, $2a+1$, and $2a+2$. None of these are counterexamples, since each has a greedy solution with 1 or 2 coins. 

Because the process found no minimum counterexample, it must be the case that $(1, 2, 3, a, a+1, 2a)$ is orderly for $a\geq 5$, as claimed.

Now let $\mathcal{B}$ and $\mathcal{B}'$ denote coin systems in the form of (b) and (c) in the statement of the theorem, respectively. Theorems~\ref{OPT} and~\ref{three} imply that the 4-value prefix coin systems of $\mathcal{B}$ and $\mathcal{B}'$ are both  totally orderly. Furthermore, the 5-value prefix coin systems for $\mathcal{B}$ and for $\mathcal{B}'$ have counterexample $2b$, so they are not orderly. Hence, all that remains is to show that the 6-value coin systems $\mathcal{B}$ and $\mathcal{B}'$ are orderly. 

Using Theorem~\ref{Pearson}, it follows that $2b$ is the minimum counterexample of the 5-value prefix coin systems of both $\mathcal{B}$ and $\mathcal{B}'$. Thus the 5-value prefix coin systems are both tight. Now suppose that $\mathcal{B}$ or $\mathcal{B}'$ is not orderly. Then by Theorem~\ref{two-not-orderly}, it follows that the coin system has a counterexample of the form $c_i + c_j$ for some coin values $c_i$ and $c_j$ in $\mathcal{B}$ or in $\mathcal{B}'$, where $1\leq i\leq j\leq 5$. However, a straightforward check of these values reveals that $c_i + c_j$ is never a counterexample. Hence it must be the case that the coin systems $\mathcal{B}$ and $\mathcal{B}'$ are orderly, as claimed.
\end{proof}
\vspace{.2cm}

\subsection{Coin systems with pattern: $+++--+$}~\label{--+}

In this section, we prove part (2) of Theorem~\ref{main}. We first prove a lemma.

\begin{lemma}\label{start}
If $\mathcal{C} = (1,c_2,c_3,c_4,c_5,c_6)$ has pattern $+++--+$, then the coin system has one of the following forms:
\begin{enumerate}
    \item $(1,a,2a-1,b,a+b-1,2b-1)$ where $a \geq 3$ and $b>2a-1$.
    \item $(1,a,2a, b, a+b,2b)$ where $a \geq 2$ and $b>2a$.

\end{enumerate}
\end{lemma}

\begin{proof}
Because the 6-value coin system is orderly, both the 4-coin and 5-value coin systems are tight. By Theorem~\ref{two-not-orderly}, there must exist a counterexample for the 5-value coin system $(1,c_2,c_3,c_4,c_5)$ of the form $c_s + c_t$ for some $1\leq s,t\leq 4$. And because the 6-value coin system is orderly, it follows that $c_6 \leq c_s + c_t \leq c_4+c_4$.

Thus there exists $0\leq i < j < k \leq 4$ such that 
\begin{align*}~\label{system}
    c_4+c_4 &= c_6 + c_i\\
    c_4+c_5 &= c_6 + c_j\\
    c_5+c_5 &= c_6 + c_k
\end{align*}
(We set $c_0 = 0$.) Subtracting adjacent equations, we find
$$c_5-c_4 = c_k - c_j = c_j - c_i.$$

The possible values for $j$ are: $1, 2,$ and $3$. We will rule out the possibility that $j=1$ or $j=3$ first, and then we will study the case $j=2$.

{\bf Case 1: $\boldsymbol{j = 1}$}. If $j=1$, it follows that $i = 0$ and hence $c_k = 2c_1 - c_0 = 2$. Thus it must be the case that $k=2$. Substituting the values $(i,j,k) = (0,1,2)$ into the above system of equations, we find the following relationships: $c_5 = c_4+1$ and $c_6 = 2c_4$. So the coin system is $(1,2,c_3, c_4, c_4+1, 2c_4)$. Since the coin system is orderly, the value $c_3+c_4$ should have a greedy solution with just 2 coins. Hence $(c_3+c_4) - (c_4+1)$ must be a value in the coin system. This forces either $c_3 = 2$ or $c_3=3$. The former option leaves no possible value for $c_2$. For the latter option, the first three coin values are $(1,2,3)$. Thus, no matter the value for $c_4$, the 4-value coin system $(1,2,3,c_4)$ is orderly, which is a contradiction. Hence $j=1$ is not possible. 

{\bf Case 2: $\boldsymbol{j = 3}$}. In this case, we must have $k=4$. Considering the possibilities $i=0,1,2$, we know $c_4 - c_3$ equals one of the elements in the set $\{c_3, c_3-1, c_3-c_2\}$. This forces the value of $c_4$ to be one of the following: $2c_3$ or $2c_3-1$ or $2c_3-c_2$. However, using the One Point Theorem, if the 3-value coin system $(1,c_2,c_3)$ is orderly, the resulting coin system $(1,c_2,c_3,c_4)$ is orderly as well, if $c_4$ is any of the options above. Therefore, it is impossible to have $j=3$.

{\bf Case 3: $\boldsymbol{j = 2}$}.  In this case, the possibilities for $i$ and $k$ and the subsequent equation $c_2 - c_i=c_k - c_2$ are:
\begin{enumerate}
\begin{multicols}{2}
     \item $i = 0$, $k = 3$, and $c_2 = c_3 - c_2$
     \item $i = 1$, $k = 3$, and $c_2-1 = c_3 - c_2$
    \item $i = 0$, $k = 4$, and $c_2 = c_4 - c_2$
    \item $i = 1$, $k = 4$, and $c_2-1 = c_4 - c_2$
\end{multicols}
\end{enumerate}
Option (1) above implies that $c_3 = 2c_2$. Plugging the values for $i, j, k$ into the system of equations, we also have $c_5 = c_4+c_2$ and $c_6 = 2c_4$. Thus the coin system is in the form of the second possibility in the statement of the lemma.

Option (2) above implies that $c_3 = 2c_2 - 1$. Moreover, observe that $c_3 = 2c_2 - 1$ is possible only if $c_2\geq 3$. For, if $c_2 = 2$, then there is no value of $c_4$ for which the 4-value coin system is not orderly. Plugging the values for $i, j, k$ into the system of equations, we also have $c_5 = c_4+c_2-1$ and $c_6 = 2c_4-1$. Thus the coin system is in the form of the first possibility in the statement of the lemma. 

Options (3) and (4) cannot occur. For, (3) and (4) imply that $c_4 = 2c_2$ or $c_4 = 2c_2-1$, respectively. On the other hand, we know from Theorem~\ref{firstjump} that $c_4\geq 3c_2 - 2$. Applying this inequality, we conclude that $c_2\leq 2$ and $c_2\leq 1$, respectively. The former option implies $c_2 = 2$ and $c_4 = 4$, which forces $c_3= 3,$ and the 4-value coin system is necessarily orderly, which is a contradiction. The latter option ($c_2\leq 1$) is impossible, since $c_2>1$.  Thus neither (3) nor (4) above are possible. This proves the lemma.
\end{proof}

Now we are ready to prove part (2) of Theorem~\ref{main}. We restate the result here for clarity.

\begin{theorem}\label{prove--+}
A coin system has pattern $+++--+$ if and only if it is one of the following:
\begin{enumerate}
    \item $\mathcal{C}=(1,a,2a-1,m(2a-1)-a+1,m(2a-1),(2m-1)(2a-1))$\\ where $1<m<a$.
    \item $\mathcal{C'}=(1,a,2a,m(2a-1)-a+1,m(2a-1)+1,(2m-1)(2a-1)+1)$\\ where $1<m \leq a$.
\end{enumerate}
\end{theorem}

\begin{proof}

First we prove that any coin system with the pattern $+++--+$ is in the form of $\mathcal{C}$ or $\mathcal{C}'$ as stated in the theorem. From the previous lemma, we know that the coin system can be in two different forms. First consider the form: $\mathcal{B}=(1,a,2a-1,b,a+b-1,2b-1)$ where $a \geq 3$ and $b>2a-1$. From the One Point Theorem, we know that the coin system $(1,a,2a-1,b)$ has a counterexample of the form $m(2a-1)$ where $m = \lceil \frac{b}{2a-1}\rceil$. 

We claim that this number $m(2a-1)$ is not a counterexample of the 5-value prefix coin system of $\mathcal{B}$. Suppose for contradiction that $m(2a-1)$ is such a counterexample. Because the 6-value coin system is orderly, we know that
$$c_6 = 2b-1 \leq m(2a-1) = \lceil \tfrac{b}{2a-1}\rceil(2a-1)\leq 
(\tfrac{b}{2a-1} + 1)(2a-1) = b + 2a-1.$$
This implies that $b\leq 2a$. Since we already know $b>2a-1$, we conclude $b=2a$. By Theorem~\ref{firstjump}, we know that the differences between adjacent coin values are at least $a-1$, and because the third and fourth coin values differ by 1, we have $a=2$, and the coin system is $(1,2,3,4,5,7)$. However, this coin system is totally orderly, which is a contradiction. Hence $m(2a-1)$ is not a counterexample for the 5-value coin system of $\mathcal{B}$.

Because $m(2a-1)$ is a counterexample for the 4-value prefix coin system, if we subtract the value of the 4th coin (which has value $b$) from $m(2a-1)$, the remaining value must use more than $m-1$ coins in its greedy solution. Whereas, because $m(2a-1)$ is {\em not} a counterexample for the 5-value prefix coin system, if we subtract the value of the 5th coin (which has value $a+b-1$) from $m(2a-1)$, the remaining value uses at most $m-1$ coins in its greedy solution. 

Letting $z = m(2a-1) - (a+b-1)$ and $y=m(2a-1) - b$, we can express the above statements as an inequality: $grd_\mathcal{B}(z) \leq m-1 < grd_\mathcal{B}(y)$.
Moreover, observe that the value of $z$ is small:
$$z=m(2a-1) - (a+b-1) \leq (\tfrac{b}{2a-1} + 1)(2a-1) - (a+b-1) = a$$
Also, notice that $y = z+ a-1.$

Collecting these statements together, we are looking for a value $z$ with two relationships: (1) $z\in [0,a]$ and (2) $grd_\mathcal{B}(z)\leq m-1< grd_\mathcal{B}(z+a-1)$. Let's take a look at the function $grd_\mathcal{B}(x)$ at values from $x=0$ up to $x=2a-1$.

\begin{equation*}
 grd_\mathcal{B}(x) = 
\begin{cases}
   
        x & \text{if } 0\leq x \leq a-1\\
        x-(a-1) & \text{if } a\leq x \leq 2a-2\\
        1 & \text{if } x=2a-1
    \end{cases}
    \end{equation*}
Observe that for all $1\leq x \leq a$, the function satisfies $grd_\mathcal{B}(x) = grd_\mathcal{B}(x+a-1)$. So $z$ cannot be any integer from 1 to $a$, because the inequality (2) given above is not satisfied. As $z\in [0,a]$, the only possible value for $z$ is $z = 0$. Replacing $z=0$ in the original expressions for $y$ and $z$, we find the values of $c_4, c_5,$ and $c_6$.
\begin{align*}
c_4 &= b = m(2a-1) -a+1\\
c_5 &= a+b-1 = m(2a-1)\\
c_6 &= 2b-1 = (2m-1)(2a-1)
\end{align*}
Keep in mind that by the definition of $m$, we know $m>1$, and because $m-1<grd_\mathcal{B}(z+a-1)= grd_\mathcal{B}(a-1) = a-1$, we also have that $m<a$.
This proves that the coin system $\mathcal{B}$ has the form of $\mathcal{C}$ in the statement of the theorem. 

Now consider the second coin system form in Lemma~\ref{start}: $$\mathcal{B'}=(1,a,2a,b,a+b,2b)$$ where $a \geq 2$ and $b>2a$. We will show that $\mathcal{B'}$ has the form of $\mathcal{C'}$ given in the statement of the theorem. Initially, the argument for $\mathcal{B'}$ mirrors what we just did for $\mathcal{B}$, but then it diverges in subtle ways, so we start from the beginning. 

From the One Point Theorem, we know that the coin system $(1,a,2a,b)$ has a counterexample of the form $2am$ where $m = \lceil \tfrac{b}{2a}\rceil$. Suppose that this value $2am$ is also a counterexample of the 5-value prefix coin system of $\mathcal{B}'$. Since the 6-value coin system $\mathcal{B}'$ is orderly, we know that
$$c_6 = 2b \leq 2am =(2a) \lceil \tfrac{b}{2a}\rceil\leq 
(2a)(\tfrac{b}{2a} + 1) = b + 2a.$$
This implies that $b\leq 2a$, a contradiction. We conclude that $2am$ is not a counterexample for the 5-value prefix coin system of $\mathcal{B}'$.

 Therefore, starting with the value $2am$, if we subtract the value of the 4th coin, the remaining value must require more than $m-1$ coins using the greedy algorithm. Whereas, starting with the value $2am$, if we subtract the value of the 5th coin, the remaining value must require at most $m-1$ coins using the greedy algorithm. Letting $z = 2am - (a+b)$ and $y=2am - b$, we can express the thought above as an inequality: $grd_\mathcal{B'}(z) \leq m-1< grd_\mathcal{B'}(y)$.
Also observe that $z \in [0,a]$ because:
$$0\leq z=2am - (a+b) \leq (\tfrac{b}{2a} + 1)(2a) - (a+b) = a.$$
Moreover, observe that $y = z+ a.$
Collecting these statements together, we are looking for an integer $z$ with two relationships: (1) $z\in [0,a]$ and (2) $grd_\mathcal{B'}(z) \leq m-1< grd_\mathcal{B'}(z+a)$. Let us take a look at the function $grd_\mathcal{B'}(x)$ at values from $x=0$ to $x=2a$, which is the same as the previous case except at the last two values ($x = 2a-1$ and $x = 2a$).

\begin{equation*}
 grd_\mathcal{B'}(x) = 
    \begin{cases}
        x & \text{if } 0\leq x \leq a-1\\
        x-(a-1) & \text{if } a\leq x \leq 2a-1\\
        1 & \text{if } x=2a
    \end{cases}
\end{equation*}
Observe that $grd_\mathcal{B'}(x) =  grd_\mathcal{B'}(x+a)-1$ for all $x\in [0,a-1]$. This implies that if $z$ is any integer in the interval $[0,a-1]$, the inequality (2) above will be satisfied. In addition, if $z$ is any value on this interval, the inequality from (2) above simplifies to $z \leq m-1< z+1$. Thus $z=m-1$. 

Replacing $z=m-1$ in the original expressions for $y$ and $z$, we find the expression for $c_4$, $c_5$, and $c_6$ as follows. 
\begin{align*}
c_4 &= b =2am - (z+a) =  m(2a-1) -a +1\\
c_5 &= a+b = m(2a-1)+1 \\
c_6 &= 2b = 2(m(2a -1) -a +1) = (2m-1)(2a-1)+1
\end{align*}
Notice that by the construction of $m$, we know that $m>1$, and since $m-1=grd_\mathcal{B'}(z) \leq a-1,$ it follows that $m\leq a$. So then $\mathcal{B'}$ has the same form as $\mathcal{C'}$, as desired.

Conversely, we must prove that the coin systems $\mathcal{C}$ and $\mathcal{C'}$ in the theorem's statement have the pattern $+++--+$. We prove a stronger result, in fact. In Theorem~\ref{E}, we prove that an infinite family of coin systems denoted $\mathcal{E}$ (for which $\mathcal{C}$ is a special case) has the pattern $+++-\cdots-+$. In Theorem~\ref{F}, we prove that an infinite family of coin system $\mathcal{F}$ (for which $\mathcal{C'}$ is a special case) has the pattern $+++-\cdots-+$. We thus delay this argument to the next section.
\end{proof}

\section{Orderly coin systems with $n$ coin values}\label{extensions}
We now study orderly coin systems with $n$ coin values.
The only $+/-$ pattern of length $n$ which cannot be characterized via first characterizing orderly coin systems with $i$ values for some $3<i<n$ is the pattern where each  prefix coin system of $i$ values where $3<i<n$ fails to be orderly. Such a coin system has the pattern $+++-\cdots -+$. We focus on these coin systems.

In Section 6 of~\cite{Adamaszek}, the authors found a single coin system with pattern $+++-\cdots -+$ and with $n$ coin values for each $n=3r$ where $r\geq 2$ and for each $n=3r+2$ where $r\geq 1$.  Interestingly, the authors wrote that they suspected that there are no coin systems with $3r+1$ coin values with this pattern where $r\geq 2$, but they were unable to prove their conjecture (see the discussion on page 60 of~\cite{Adamaszek}). 

Using a computer program, we continued and extended this search for coin systems with pattern $+++-\cdots -+$. A detailed discussion of this computer search is given in Section~\ref{future}. Many observed examples have a particular structure. We call these coins systems {\bf fixed gap coin systems}, which we define below. 
 The examples we observed fit into three distinct infinite families: $\mathcal{D}, \mathcal{E}$, and $\mathcal{F}$. The coin systems in family $\mathcal{D}$ each have $3r+2$ different coin values where $r\geq 1$, while the coin systems in $\mathcal{E}$ and $\mathcal{F}$ each have $3r$ different coin values where $r\geq 2$. In line with the conjecture of~\cite{Adamaszek}, we found no coin systems with $3r+1$ coin values where $r\geq 2$. We prove that these three infinite families of coin systems $\mathcal{D}, \mathcal{E}$, and $\mathcal{F}$ have the pattern $+++-\cdots-+$.

 A fixed-gap coin system with $n$ coin values starts off with two initial coin values $(c_1,c_2)=(1,x)$. Then all the remaining coin values are generated in two stages. First choose two positive integers $\Delta_1$ and $\Delta_2$ where $\Delta_1\neq \Delta_2$.
 \begin{itemize}
 \item {\bf Stage 1:} Alternately add $\Delta_1$ and $\Delta_2$ to the previous coin value to produce the next coin value. We continue this until we have $\ell$ coin values for some $3<\ell<n$. 
 \item {\bf Stage 2:} For coin values $c_{\ell+1}$ through $c_n$, add $\Delta_1+\Delta_2$ to the previous coin value to produce the next coin value.  
 \end{itemize}
We make this process precise as follows. 

\begin{definition}\label{fixedgap} Let $n$, $\ell$, $x$, $\Delta_1$, and $\Delta_2$ be positive integers with $3<\ell < n$, $x\geq 2$, and   $\Delta_1 \neq \Delta_2$. The coin system $\mathcal{C} = (1, c_2, c_3, \ldots, c_n)$ is a {\bf fixed gap coin system} if $c_i$ are defined as follows.
\begin{equation}\label{C}
 c_i = 
\begin{cases}
        x & \text{if } i=2\\
        c_{i-1}+\Delta_1  & \text{if $i$ is odd and } 3\leq  i \leq \ell\\
        c_{i-1}+\Delta_2  & \text{if $i$ is even and } 3< i \leq \ell\\
        c_{i-1}+\Delta_1+\Delta_2 & \text{if } \ell<i\leq n
\end{cases}
\end{equation}

 \end{definition} 

For most fixed-gap coin system $\mathcal{C}$ as above, the coin system has several consecutive prefix coin systems that fail to be orderly. We prove this in the following lemma.

 \begin{lemma}\label{Enotorderly}
Let $\mathcal{C}$ be a fixed-gap coin system. If $\ell$ is odd, 
then each prefix coin system of $\mathcal{C}$ of length $k$ with $5 \leq k\leq \frac{1}{2}(3\ell-5)$ is not orderly.
If $\ell$ is even
and $x+\Delta_1 > \Delta_2+1$, then each prefix coin system of $\mathcal{C}$ of length $k$ with $5 \leq k \leq \frac{1}{2}(3\ell -4)$ is not orderly.
 \end{lemma}
 
\begin{proof}
Throughout this proof, we use the same strategy multiple times to show that a given coin system $(1, c_2, c_3, \ldots, c_k)$ is not orderly. The outline is as follows. We prove:
\begin{enumerate}
\item  $2c_{k-1} - c_k >0$, and  
\item $2c_{k-1}-c_k \neq c_i$ for any $i<k$. 
\end{enumerate}
This implies that $2c_{k-1}$ is a counterexample because on one hand the greedy solution has more than 2 coins, while the optimal solution has 2 coins.

To begin, we prove that the prefix coin system $(1, x, c_3, \ldots, c_k)$ fails to be orderly for all $k$ such that $5 \leq k \leq \ell$.

Suppose that $k$ is even.  
We first show $2c_{k-1}-c_k$ is positive as follows:
\begin{equation*}
    2c_{k-1} - c_k = 2c_{k-1} - (c_{k-1} +\Delta_2) = c_{k-1} - \Delta_2 = (c_{k-3} +\Delta_1+\Delta_2) - \Delta_2= c_{k-3}+\Delta_1 >0
\end{equation*}
Now we show that $2c_{k-1}-c_k$ is not a coin value $c_i$ for any $i$. Observe that if $\Delta_1<\Delta_2$, then the value $2c_{k-1} - c_{k}=c_{k-3}+\Delta_1$ falls between consecutive coin values:
\begin{equation}\label{+y}
c_{k-3} < c_{k-3}+\Delta_1 < \underbrace{c_{k-3} + \Delta_2}_{c_{k-2}} 
\end{equation}
And if $\Delta_2< \Delta_1$, we similarly have:
\begin{equation}\label{+y2}
\underbrace{c_{k-3} + \Delta_2}_{c_{k-2}} < c_{k-3}+\Delta_1< \underbrace{c_{k-3} + \Delta_1 + \Delta_2}_{c_{k-1}}. 
\end{equation}
So no matter the relative sizes of $\Delta_1$ and $\Delta_2$, the value $2c_{k-1} - c_{k}$ falls between consecutive coin values, and therefore the value is not itself a coin value. We conclude that $2c_{k-1}$ is a counterexample for the prefix coin system  $(1,c_2,c_3, \ldots, c_k)$ where $k$ is even and $5 <k\leq \ell$.

Now suppose that $k$ is odd and $5 \leq k\leq \ell$. 
Similar to before, we show that $2c_{k-1}$ is a counterexample for the prefix coin system $(1,c_2,c_3, \ldots, c_{k})$. 
First, we simplify the expression $2c_{k-1} - c_{k}$ and conclude that it must be positive:
$$
2c_{k-1} - c_{k} = 2c_{k-1} - (c_{k-1}+\Delta_1) = c_{k-1} -\Delta_1 =(c_{k-3} +\Delta_1+\Delta_2)-\Delta_1 = c_{k-3} +\Delta_2 >0
$$
Similar to the previous case, observe that if $\Delta_1<\Delta_2$, then the value $2c_{k-1} - c_{k}=c_{k-3} +\Delta_2$ falls between consecutive coin values:
\begin{equation}\label{+z}
\underbrace{c_{k-3} + \Delta_1}_{c_{k-2}} < c_{k-3} +\Delta_2 < \underbrace{c_{k-3} + \Delta_1 + \Delta_2}_{c_{k-1}} 
\end{equation}
And if $\Delta_2<\Delta_1$, we have:
\begin{equation}\label{+z2}
c_{k-3} < c_{k-3} +\Delta_2 < \underbrace{c_{k-3} + \Delta_1}_{c_{k-2}}. 
\end{equation}
So in either case, the value $2c_{k-1} - c_{k}=c_{k-3} +\Delta_2$ falls between consecutive coin values, and therefore the value is not itself a coin value. We conclude that $2c_{k-1}$ is a counterexample for the prefix coin system  $(1,c_2,c_3, \ldots, c_k)$ where $k$ is odd and $5 \leq k\leq \ell$.

It remains to show that all prefix subcurrencies $(1,c_2,c_3, \ldots, c_k)$ where $\ell <k\leq t$ also fail to be orderly, where $t=\frac{1}{2}(3\ell-5)$ if $\ell$ is odd and $t = \frac{1}{2}(3\ell-4)$ if $\ell$ is even. 

Suppose $\ell = 2r+1$ for some $r \geq 2$. We show that $(1,c_2,c_3, \ldots, c_k)$ where $\ell <k \leq 3r-1$ fail to be orderly by showing $2c_{2r}$ is a counterexample. 
\begin{align*}
    2c_{2r}-c_k &= 2c_{2r} - (c_{2r+1} + (k-2r-1)(\Delta_1+\Delta_2)) \\
    &= 2c_{2r} - (c_{2r} +\Delta_1+(k-2r-1)(\Delta_1+\Delta_2))\\
    &=c_{2r} - \Delta_1 - (k-2r-1)(\Delta_1+\Delta_2)\\
    &= c_{2r - 2(k-2r)} +(k-2r)(\Delta_1+\Delta_2) - \Delta_1 - (k-2r-1)(\Delta_1+\Delta_2)\\
    &= c_{6r - 2k} +\Delta_2
\end{align*}
We previously proved that a value of form $c_{2i} +\Delta_2$ where $1\leq i \leq r-1$ falls between adjacent coin values (Inequalities~\ref{+z} and~\ref{+z2}). Therefore, $2c_{2r}$ is a counterexample for the prefix coin system $(1,c_2,c_3, \ldots, c_k)$ where $\ell <k \leq  3r-1$ and $\ell$ is odd.

Finally, suppose $\ell$ is even. That is, $\ell = 2r+2$ for some $r \geq 1$. And suppose $x+\Delta_1 > \Delta_2+1$. We show that $(1,c_2,c_3, \ldots, c_k)$ where $\ell <k \leq  3r+1$ fails to be orderly by showing $2c_{2r+1}$ is a counterexample. First we do the special case that $k=3r+1$.
\begin{align*}
    2c_{2r+1}-c_{3r+1} &= 2c_{2r+1} - (c_{2r+2} + (r-1)(\Delta_1+\Delta_2)) \\
    &= 2c_{2r+1} - (c_{2r+1} +\Delta_2+(r-1)(\Delta_1+\Delta_2))\\
    &=c_{2r+1} - \Delta_2 - (r-1)(\Delta_1+\Delta_2)\\
    &= c_{2r+1 - 2(r-1)} +(r-1)(\Delta_1+\Delta_2) - \Delta_2 - (r-1)(\Delta_1+\Delta_2)\\
    &= c_{3} -\Delta_2\\
    &= x+\Delta_1-\Delta_2
\end{align*}
By our assumption, $x+\Delta_1 > \Delta_2+1$. Thus, $1<x+\Delta_1 - \Delta_2<c_3$ and moreover $x+\Delta_1 - \Delta_2$ cannot equal $c_2=x$. Therefore, this value $x+\Delta_1 - \Delta_2$ is not a coin value and $2c_{2r+1}$ is a counterexample for the prefix coin system $(1, c_2, c_3, \ldots, c_{3r+1})$.

Now suppose $\ell <k <  3r+1$. We prove that this prefix coin system is not orderly by showing $2c_{2r+1}$ is a counterexample.
\begin{align*}
    2c_{2r+1}-c_k &= 2c_{2r+1} - (c_{2r+2} + (k-2r-2)(\Delta_1+\Delta_2)) \\
    &= 2c_{2r+1} - (c_{2r+1} +\Delta_2+(k-2r-2)(\Delta_1+\Delta_2))\\
    &=c_{2r+1} - \Delta_2 - (k-2r-2)(\Delta_1+\Delta_2)\\
    &= c_{2r+1 - 2(k-2r-2)} +(k-2r-2)(\Delta_1+\Delta_2) - \Delta_2 - (k-2r-2)(\Delta_1+\Delta_2)\\
    &= c_{6r - 2k+5} -\Delta_2\\
    &= c_{6r - 2k+3} +\Delta_1
\end{align*}
We previously proved that a value of the form $c_{2i+1} +\Delta_1$ where $1\leq i< r-1$ falls between adjacent coin values (Inequalities~\ref{+y} and~\ref{+y2}.). Therefore, $2c_{2r+1}$ is a counterexample for the coin system $(1,c_2,c_3, \ldots, c_k)$ where $\ell <k <  3r+2$.
\end{proof}

Now we define the three infinite families of fixed-gap coin systems (denoted by $\mathcal{D}, \mathcal{E}$, and $\mathcal{F}$, respectively) and prove that they have pattern $+++-\cdots-+$. 
\begin{theorem}\label{D}
Let $r\geq1$ and $a \geq 2$. Define $\mathcal{D} = (1, d_2, \ldots, d_{3r+2})$ as follows:
$$ d_k = 
\begin{cases}
        2 & \text{if } k=2\\
        d_{k-1}+a  & \text{if $k$ is odd and } 3\leq  k  < 2r+2\\
        d_{k-1}+1  & \text{if $k$ is even and } 3< k \leq 2r+2\\
        d_{k-1}+a+1 & \text{if } 2r+2 <k\leq 3r+2
 \end{cases}
    $$
The coin system $\mathcal{D}$ has pattern $+++-\cdots -+$. 
\end{theorem}
\begin{example}
 For $r=3$ and $a=3$, the coin system is 
 $$\mathcal{D} = (1,2,5,6,9,10,13,14,18,22,26).$$
 \end{example}
\begin{proof}
The 3-coin prefix coin system $(d_1,d_2,d_3)=(1,2,a+2)$ is totally orderly for all values of $a\geq 2$ by Theorem~\ref{three} where $d_3-d_2 = a$. The 4-coin prefix coin system $(1,2,a+2, a+3)$ is not orderly because $2a+4$ is a counterexample. From Lemma~\ref{Enotorderly}, we know all prefix coin systems of length $i$ where $5 \leq i \leq 3r+1$ are not orderly. 

It remains to show the full coin system $\mathcal{D}$ is orderly. 
Suppose that the coin system is not orderly. We use Theorem~\ref{Pearson} to find the smallest counterexample. First, for all $k\geq 2$, we express the greedy solution for $d_k - 1$ as a vector with respect to the coin system $\mathcal{D}$. In the vectors below, an asterisk (*) denotes that the value is in the $k-1$ position.
For $k$ even and $2\leq k\leq 2r+2$, we have
$$d_k - 1 = d_{k-1}=(0, 0, \ldots, 1^*, \ldots, 0)_\mathcal{D}. $$
For $k$ odd and $2 < k < 2r+2$, we have
$$d_k - 1 = d_{k-1}+a-1= (x_1, x_2, 0, \ldots, 1^*, \ldots, 0)_\mathcal{D} $$
where $x_1$ and $x_2$ are integers such that  $a-1 = 2x_2+x_1$, $x_2\geq 0$ and $x_1 \in\{0,1\}$.
For $2r+1< k\leq 3r+2 $, 
$$d_k - 1 = d_{k-1}+a= (y_1, y_2, 0, \ldots, 1^*, 0, \ldots, 0)_\mathcal{D}$$
where $y_1$ and $y_2$ are integers such that $a = 2y_2+y_1$, $y_2\geq 1$, and $y_1 \in \{0,1\}$. 

Now for each $1\leq p< k-1$, we set the first $p$ entries of the vectors above to 0 and add 1 to the $p+1$ entry. The set of the resulting vectors gives all the candidates for the smallest counterexample. When $p=1$, the resulting vector in each case represents one of the following values: $d_i$, $d_i +1$, or $d_i+2$. None of these values are counterexamples, since their greedy solutions have either 1 or 2 coins. When $p\geq 2$, the resulting vector in each case represents a value of the form $d_i + d_j$ where $3\leq i,j < 3r+2$.   We compute the greedy solution for  $d_i+d_j$ for all $i,j$. 

First we consider the sum $d_i + d_j$ where both of the coin values belong to Stage 1 of the construction. That is, $3\leq i,j\leq 2r+2$.
\begin{equation*}
 d_i + d_j = 
\begin{cases}
        d_2 + d_{i+j-2} & \text{if $i$ and $j$ are even and } i+j \leq 2r+4\\
        d_2 + d_{(i+j+2r)/2} & \text{if $i$ and $j$ are even and } i+j > 2r+4  \\
        d_{i+j} & \text{if $i$ and $j$ are odd and } i+j \leq 2r+2\\
        d_{r+1+(i+j)/2} & \text{if $i$ and $j$ are odd and } i+j > 2r+2   \\
        d_1+d_{i+j-1} & \text{if $i$ is odd and $j$ is even and  } i+j \leq 2r+2\\
        d_1 + d_{(i+j+2r+1)/2} & \text{if $i$ is odd and $j$ is even and  } i+j > 2r+2 
\end{cases}
\end{equation*}

Next we compute the greedy solution for $d_i+d_j$ where at least one coin value (say, $d_j$) belongs to Stage 2 of the construction. That is, $2r+2 <j < 3r+2$. Notice that we need not consider the case $j = 3r+2$ because in that case, the greedy solution of $d_i + d_{3r+2}$ is simply itself.
\begin{equation*}
 d_i + d_j = 
\begin{cases}       
        d_{3r+2} + d_{i+2j-6r-4} & \text{if $2<i<2r+2$ is odd and } \tfrac{1}{2}(i-1)+j \geq 3r+2\\
        d_1 + d_{j+(i-1)/2} & \text{if $2<i<2r+2$ is odd and } \tfrac{1}{2}(i-1)+j < 3r+2 \\
         d_{3r+2} + d_{i+2j-8r-8} & \text{if $2\leq i\leq 2r+2$ is even and } \tfrac{1}{2}(i-2)+j \geq 3r+2\\
        d_2 + d_{j+(i-2)/2} & \text{if $2\leq i\leq 2r+2$ is even and } \tfrac{1}{2}(i-2)+j < 3r+2\\
         d_{3r+2} + d_{2(i+j-4r-3)} & \text{if $2r+2 < i < 3r+2$ and } i+j \leq 5r+4\\
        d_{3r+2}+ d_{i+j-3r-2} & \text{if $2r+2 < i < 3r+2$ and } i+j > 5r+4  
\end{cases}
\end{equation*}
None of these values $d_i + d_j$ are counterexamples because each greedy solution has only 1 or 2 coins. Therefore, the coin system $\mathcal{D}$ does not have a minimum counterexample, and the coin system  is orderly.
\end{proof}

\begin{theorem}\label{E}
Let $r\geq 2$ and $1<m < a$. Define $\mathcal{E} = (1, e_2, \ldots, e_{3r})$ as follows:
$$
 e_k = 
\begin{cases}
        a & \text{if } k=2\\
        e_{k-1}+a-1  & \text{if $k$ is odd and } 3\leq  k \leq 2r+1\\
        e_{k-1}+(m-1)(2a-1)-(a-1) & \text{if $k$ is even and } 3< k < 2r+1\\
        e_{k-1}+(m-1)(2a-1) & \text{if } 2r+1 <k\leq 3r
\end{cases}
$$
The coin system $\mathcal{E}$ has pattern $+++-\cdots -+$. 
\end{theorem}
\begin{example}
In the special case where $r=2$, we recover the coin system with six coins studied in Theorem~\ref{prove--+}(1): $$(1,a,2a-1,m(2a-1)-a+1,m(2a-1),(2m-1)(2a-1))$$ where $1<m<a$.
\end{example}
\begin{example}
 For $r=4$, $m=3$, and $a=4$, the coin system is 
 $$\mathcal{E} = (1,4,7,18,21,32,35,46,49,63,77,91).$$
 \end{example}
\begin{proof}

The 3-value prefix coin system $(e_1,e_2,e_3)=(1,a,2a-1)$ is totally orderly by Theorem~\ref{three} where $e_3-e_2 = a-1$. 

The 4-coin prefix coin system $(1,a,2a-1,m(2a-1)-(a-1))$ where $1<m<a$ is not orderly because $m(2a-1)$ is a counterexample. By Lemma~\ref{Enotorderly}, we know all prefix coin systems of length $i$ where $5 \leq i \leq 3r-1$ are not orderly. 

Now we show that the full coin system $\mathcal{E}$ is orderly using Theorem~\ref{Pearson}. First, for all $k\geq 2$, we express the greedy solution for $e_k - 1$ as a vector with respect to the coin system $\mathcal{E}$. In the vectors below, the asterisk (*) denotes that the value is in the $k-1^{st}$ entry.

\begin{equation}\label{evectors}
e_k - 1 =
\begin{cases}
(a-1,0,\ldots,0)_\mathcal{E} & \text{if $k=2$}\\
(a-2, 0, \ldots, 1^*, \ldots, 0)_\mathcal{E} &
\text{if $k$ is odd, $3 \leq k\leq 2r+1$}\\
(a-1, 0,m-1,0,\ldots, 0)_\mathcal{E} & \text{if $k=4$}\\
(a-1, 0,m-2,0 \ldots, 1^*, \ldots, 0)_\mathcal{E} & \text{if $k$ is even, $6 \leq k < 2r+1$}\\
(a-2, 1, m-2,0, \ldots, 1^*, \ldots, 0)_\mathcal{E} &
\text{if $2r+1< k\leq 3r $}
\end{cases}
\end{equation}

Now for each $1\leq p < k-1$, we set the first $p$ entries of the vectors above to 0 and add 1 to the $p+1$ entry. In the first vector (when $k=2$), there is nothing to check since in this case $1\leq p$ but also $p< k-1=1$. In most cases, the resulting vector corresponds to the value  $e_i + e_j$ where $2\leq i,j< 3r$ (specifically, this is the case for the second vector and all values of $p$, and for the final two vectors in the case that $3\leq p<k-1$). So we must check whether a value of this form is a counterexample. 

Accordingly, we find the greedy solution for $e_i + e_j$ for all $2\leq i,j< 3r$. First, suppose $2\leq i,j\leq 2r+1$ and $m>2$.
\begin{equation*}
 e_i + e_j = 
\begin{cases}
        e_{i+j-1} +e_1 & \text{if $i$ and $j$ even and } i+j \leq 2r+2 \\
        e_{(i+j)/2+r}+e_1 & \text{if $i$ and $j$ even and } i+j > 2r+2 \\
        e_{i+j-3} +e_3 & \text{if $i$ and $j$ odd and } i+j \leq 2r+4 \\
        e_{(i+j)/2+r-1}+e_3 & \text{if $i$ and $j$ odd and } i+j > 2r+4    \\
        e_{i+j-2} +e_2 & \text{if $i$ is odd and $j$ is even and  } i+j \leq 2r+4 \\
        e_{r+(i+j-1)/2} + e_2 & \text{if $i$ is odd and $j$ is even and  } i+j > 2r+4
\end{cases}
\end{equation*}
Next suppose $2r+1< j<3r$ and $m>2$. Then
\begin{equation*}
 e_i + e_j = 
\begin{cases}         
       e_{j+(i-3)/2} +e_3 & \text{if $2< i\leq 2r+1$ is odd and } i+2j \leq 6r+3 \\
        e_{3r} + e_{i+2j-6r} & \text{if $2< i\leq 2r+1$ is odd and } i+2j > 6r+3\\
       e_{j-1+i/2} +e_2  & \text{if $2\leq i\leq 2r+1 $ is even and } i + 2j \leq 6r+2\\
       e_{3r} + e_{i+2j-6r} & \text{if $2\leq i\leq 2r+1$ is even and } i+2j > 6r+2\\
       e_{3r}+e_{2i+2j-8r-1} & \text{if $2r+1< i<3r$, and $i+j\leq 5r+1$}\\
      
      e_{3r}+e_{i+j-3r} & \text{if $2r+1< i<3r$, and $i+j> 5r+1$ }
\end{cases}
\end{equation*}
In the special case $m=2$, the above equations require three adjustments. If $m=2$, then $e_3 +e_k =e_{k+2}$ for $2\leq  k< 2r$  and $e_3+e_k = e_{k+1}$ for $2r+1 \leq k < 3r+1$. Thus if $m=2$, anytime we see $e_k+e_3$ above where $k\geq 2$ and $k\neq 2r$, we eliminate the sum and replace it with a single coin value as the greedy solution. 

Observe that in all the above expressions, the greedy solution for $e_i + e_j$ contains either 1 or 2 coins. Therefore, $e_i+e_j$ is not a counterexample. 

Now, we consider one by one the last three vectors from Equation~\ref{evectors} where $p=1$ and $p=2$. Begin with $e_4 - 1$. For $p = 1$ and $p=2$, respectively, we have the following two vectors, associated candidates for counterexample, and greedy solutions. 
\begin{center}
\begin{tabular}{|c | c | c |}
  \hline
   Vector representation & Candidate for counterexample & Greedy solution \\ 
  \hline
 $(0,1,m-1,0,\ldots, 0)_\mathcal{E} $ & $2am - m - a+1$ & $e_{4}$\\ 
  \hline
$(0,0,m,0,\ldots, 0)_\mathcal{E}$  &  $m(2a-1)$ & $e_{5}$ \\ 
  \hline
\end{tabular}
\end{center}
The greedy solutions above both use 1 coin, so neither of these values are counterexamples. 

Now consider $e_k - 1$ where $k$ is even, $6\leq k < 2r+1$ (see Equation~\ref{evectors} for the vector representation). For $p = 1$ and $p=2$, respectively, we have the following results. (Recall that an asterisk (*) denotes that the value is in the $k-1^{st}$ entry.)

\begin{center}
\begin{tabular}{|c | c | c |}
  \hline
   Vector representation & Candidate for counterexample & Greedy solution \\ 
  \hline
 $(0, 1,m-2,0 \ldots, 1^*, \ldots, 0)_\mathcal{E}$& $e_{k-1} + 2am-3a - m +2$ & $e_{k}$\\ 
  \hline
$ (0, 0,m-1,0 \ldots, 1^*, \ldots, 0)_\mathcal{E}$  &  $e_{k-1} + 2am - 2a-m+1$ & $e_{k+1}$ \\ 
  \hline
\end{tabular}
\end{center}
Again, the greedy solutions above both use 1 coin, so neither of these values are counterexamples. 

Finally, consider $e_k - 1$ where $2r+1< k\leq 3r $ (see Equation~\ref{evectors}). For $p=1$ and $p=2$, respectively, we have the following two vectors, candidates for counterexample, and associated greedy solutions: 
\begin{center}
\begin{tabular}{|c | c | c |}
  \hline
   Vector representation & Candidate for counterexample & Greedy solution \\ 
  \hline
$ (0, 2, m-2,0, \ldots, 1^*, \ldots, 0)_\mathcal{E}$ &   $ e_{k-1}+2a(m-1)-(m-2)$  & $e_{k}+e_1$ \\ 
 \hline
$(0, 0, m-1,0, \ldots, 1^*, \ldots, 0)_\mathcal{E}$& $e_{k-1} + (2a-1)(m-1)$  & $e_{k}$\\ 
  \hline
\end{tabular}
\end{center}

Once more, the greedy solutions above use either 1 or 2 coins, so neither of these values are counterexamples. Thus the coin system $\mathcal{E}$ is orderly.
\end{proof}

\begin{theorem}\label{F}
Let $r\geq 2$ and $1<m\leq a$. Define $\mathcal{F} = (1, f_2, \ldots, f_{3r})$ as follows:
$$
 f_k = 
\begin{cases}
        1 & \text{if } k=1\\
        a & \text{if } k=2\\
        f_{k-1}+a  & \text{if $k$ is odd and } 3\leq  k \leq 2r+1\\
        f_{k-1}+(m-1)(2a-1)-a  & \text{if $k$ is even and } 3< k < 2r+1\\
        f_{k-1}+(m-1)(2a-1) & \text{if } 2r+1 <k\leq 3r
\end{cases}
$$
The coin system $\mathcal{F}$ has pattern $+++-\cdots -+$. 
\end{theorem}
\begin{example}
In the special case where $r=2$, we recover the coin system with six coins studied in Theorem~\ref{prove--+}(2): $$(1,a,2a,m(2a-1)-a+1,m(2a-1)+1,(2m-1)(2a-1)+1)$$ where $1<m \leq a$.
\end{example}
\begin{example}
 For $r=4$, $m=3$, and $a=4$, the coin system is 
 $$\mathcal{F} = (1,4,8,18,22,32,36,46,50,64,78,92).$$
 \end{example}
\begin{proof}
The 3-coin prefix coin system $(f_1,f_2,f_3)=(1,a,2a)$ is totally orderly by Theorem~\ref{three} where $f_3-f_2 = a$. 

The 4-coin prefix coin system $(1,a,2a,(m-1)(2a-1)+a )$ where $1<m\leq a$ is not orderly because $2am$ is a counterexample.
By Lemma~\ref{Enotorderly}, we know all prefix coin systems of length $i$ where $5 \leq i \leq 3r-1$ are not orderly.  

Finally, we show that the full coin system $\mathcal{F}$ is orderly. 
Suppose that the coin system is not orderly. We use Theorem~\ref{Pearson} to find the smallest counterexample. We find the greedy solution for $f_k - 1$ for all $k\geq 2$ and express it as a vector with respect to the coin system $\mathcal{F}$. In the vectors below, the asterisk (*) denotes that the value is in the $k-1^{st}$ entry.
\begin{equation}\label{fvectors}
f_k - 1 =\
\begin{cases}
(a-1,0,\ldots,0)_\mathcal{F} & \text{if $k=2$}\\
(a-1, 0, 0, \ldots, 1^*, \ldots, 0)_\mathcal{F} &
\text{if $k$ is odd, $3 \leq k\leq 2r+1$}\\
(a-m, 0,m-1,0 \ldots, 0)_\mathcal{F} & \text{if $k=4$}\\
(a-m, 0,m-2,0 \ldots, 1^*, \ldots, 0)_\mathcal{F} & \text{if $k$ is even, $6 \leq k < 2r+1$}\\
(a-m, 1, m-2,0, \ldots, 1^*, \ldots, 0)_\mathcal{F} &
\text{if $2r+1< k\leq 3r $}
\end{cases}
\end{equation}
As before, for each $1\leq p\leq k-1$, we set the first $p$ entries of the above vectors to 0 and add 1 to the $p+1$ entry. Similar to the previous theorem, the case $k=2$ yields nothing. For the second case in Equation~\ref{fvectors}, the resulting vector has only two nonzero values (both equal to 1) for all such $p$, which corresponds to the value $f_i+f_j$. The same is true for the last two cases in Equation~\ref{fvectors} when $p \geq 3$. To settle these cases, we find the greedy solution for $f_i + f_j$ for all $2\leq i,j< 3r$. 

Suppose $2\leq i,j\leq 2r+1$ and $m>2$. The greedy solution for $f_i+f_j$ is:
\begin{equation}\label{small}
 f_i + f_j = 
\begin{cases}
        f_{i+j -1} & \text{if $i$ and $j$ even and } i+j \leq 2r+2 \\

         f_{r+ (i+j)/2} & \text{if $i$ and $j$ even and } i+j > 2r+2 \\
        
        f_{i+j-3}+f_3 & \text{if $i$ and $j$ odd and } i+j \leq 2r+4 \\
     f_{r-1+(i+j)/2}+f_3 & \text{if $i$ and $j$ odd and } i+j > 2r+4    \\
 
        f_{i+j-2}+f_2 & \text{if $i$ is even and $j$ is odd and } i+j \leq 2r+3 \\
         
        f_{r+(i+j-1)/2}+f_2 & \text{if $i$ is odd and $j$ is even and  } i+j > 2r+3
\end{cases}
\end{equation}
Suppose $2r+1< j < 3r$ and $m>2$. Then we have:
\begin{equation}\label{big}
 f_i + f_j = 
\begin{cases}        
        f_{j +(i-3)/2} + f_3 
       & \text{if $2< i\leq 2r+1$ is odd and } i + 2j \leq 6r+3 \\
        
        f_{3r} + f_{i +2j-6r} 
        & \text{if $2< i\leq 2r+1$ is odd and } i + 2j > 6r+3\\
         
        f_{j-1+i/2} + f_2 
        & \text{if $2\leq i< 2r+1 $ is even and } i + 2j \leq 6r+2\\
         
       f_{3r} + f_{i+2j-6r} 
        & \text{if $2\leq i< 2r+1$ is  even and } i + 2j> 6r+2\\

       f_{3r} + f_{2i+2j-8r-1} 
     & \text{if $2r+1< i < 3r  $ and $i+j \leq 5r+1$}\\
      f_{3r} + f_{i+j-3r} 
     & \text{if $2r+1< i < 3r $ and $i+j > 5r+1$}
\end{cases}
\end{equation}

From this, we see that if $m>2$, then $f_i + f_j$ is not a counterexample no matter the values of $2\leq i,j < 3r$, because the greedy solutions given above always have 1 or 2 coins. 

In the special case $m=2$, the above equations require three adjustments. The greedy solution of $f_k + f_3$ is $f_{k+2}+f_1$ if $3\leq k < 2r+1$. And, if $2r+1\leq k < 3r$, the greedy solution is $f_{k+1}+f_1$. So the expressions in cases 3 and 4 of Equation~\ref{small} and case 1 of Equation~\ref{big} above need appropriate adjustments.  Everything else remains as is, and again we conclude that $f_i+f_j$ is not a counterexample.

Now we consider the remaining special cases. Namely, we consider the last three cases of Equation~\ref{fvectors} where $p=1$ and $p=2$.  First we simply work out the candidates for minimum counterexample. 

We begin with $f_4-1$. (See Equation~\ref{fvectors} for the vector representation of $f_4-1$.) The vectors associated to $p=1$ and $p=2$ and the resulting candidates for counterexamples are below.

\begin{center}
\begin{tabular}{|l | c | c |}
  \hline
   Vector representation & Candidate for counterexample & Greedy solution \\ 
  \hline
$(0, 1,m-1,0 \ldots, 0)_\mathcal{F}$ & $2am - a$ & $f_4+(m-1)f_1$\\
\hline
 &   & $f_6$, if $m=a=2$  \\
$(0, 0,m,0 \ldots, 0)_\mathcal{F}$& $2am$ &  else $f_5+(m-1)f_1$ \\
\hline

\end{tabular}
\end{center}

Next, consider $f_k - 1$ where $k$ is even and $6\leq k < 2r+1$. (Again, see Equation~\ref{fvectors} for the vector representation of $f_k-1$.) The vectors associated to $p=1$ and $p=2$ and the resulting candidates for counterexamples are below. An asterisk (*) denotes that the value is in the $k-1^{st}$ entry.

\begin{center}
\begin{tabular}{|l | c | c |}
  \hline
   Vector representation & Candidate for  & Greedy solution \\ 
      & counterexample &  \\ 
  \hline
$(0, 1,m-2,0 \ldots, 1^*, \ldots, 0)_\mathcal{F}$& $f_{k-1}+2am-3a$&  $f_{k} + (m-1)f_1$ \\ 
  \hline
 & & $f_{k+2}$, if $m=a=2$,
 \\
$(0, 0,m-1,0 \ldots, 1^*, \ldots, 0)_\mathcal{F}$& $f_{k-1}+2am-2a$&  else $f_{k+1} + (m-1)f_1$

\\ 
  \hline
\end{tabular}
\end{center}

Now consider $f_k-1$ where $2r+1 < k < 3r$. The vectors associated to $p=1$ and $p=2$ and the resulting candidates for counterexamples are below. 

\begin{center}
\begin{tabular}{|l | c | c |}
  \hline
   Vector representation & Candidate for counterexample & Greedy solution \\ 
  \hline
$(0, 2, m-2,0, \ldots, 1^*, \ldots, 0)_\mathcal{F}$& $f_{k-1} + 2am-2a$ & $f_{k}+(m-1)f_1$
\\ 
  \hline
$(0, 0, m-1,0, \ldots, 1^*, \ldots, 0)_\mathcal{F}$ &$f_{k-1} + 2am -2a$ & $f_{k}+(m-1)f_1$
\\ 
  \hline
\end{tabular}
\end{center}

We now show that none of the values given above as candidates for the minimum counterexample are, in fact, counterexamples. If $m=2$, then  the greedy solutions in each of the three cases are optimal, since each solution has only 1 or 2 coins. For $m>2$, a total of $m$ coins are in the greedy solution, and we must work carefully to show that the solution is optimal. Observe that all the greedy solutions above are of the form $f_k+(m-1)f_1$ for some value $4\leq k < 3r$. So it remains to show that an optimal solution for $f_k+(m-1)f_1$ has $m$ coins. 

Suppose that there exists a $k$ such that an optimal solution $f_k+(m-1)f_1$ has fewer than $m$ coins. Take the smallest such value, which we denote by $v$. By Theorem~\ref{Chang}, an optimal solution for $v$ does not use $f_1$ or $f_k$. 

We know that an optimal solution for $v$ uses at most one coin $f_j$ where $2r+1<j<3r$ because the sum of any two such coins is larger than $f_{3r}+m-1$ and $v \leq f_{3r}+m-1$.

An optimal solution for $v$ has at most one even-indexed coin $f_j$ where $2\leq j < 2r+1$, because if there were two such coins $f_i$ and $f_j$, then they could be replaced with a single coin (either $f_{i+j-1}$ or $f_{r+(i+j)/2}$, see Equation~\ref{small}), which contradicts the optimality of the solution. 

If an optimal solution for $v$ contains an even indexed coin $f_j$ where $3<j < 2r+1$, then we can trade the coins $f_j$ and any other coin $f_i$ in the optimal solution for $f_2$ and a larger indexed coin (see Equations~\ref{small} and~\ref{big}). Therefore, without loss of generality, the optimal solution for $v$ has no even-indexed coins $f_j$ where $3< j< 2r+1$.

Next, if an optimal solution for $v$ contains two coins  $f_i$ and $f_j$ where $i$ and $j$ are odd and $3< i,j\leq 2r+1$, then we can trade them in for $f_3$ and another coin $f_r$ 
(again, see Equation~\ref{small}). Therefore without loss of generality, the optimal solution has at most one coin $f_i$ with odd index $3<i\leq 2r+1$.

If an optimal solution for $v$ has both a coin of value $f_i$ where $i$ is odd $3\leq i\leq 2r+1$ and a coin of value $f_j$ where  $2r+1<j\leq 3r$, we can exchange the pair of coins for a different pair of coins $f_q$ and $f_3$ where $2r+1<q\leq 3r$ (Equation~\ref{big}). 

Putting all of these observations together, we conclude that any optimal solution  for $v$ can be exchanged for another optimal solution fitting one of the following scenarios:
\begin{enumerate}
     \item $v= xf_3$, where $0< x \leq m-1$
     \item $v= f_2 + yf_3$, where $0< y \leq m-2$
    \item 
    $v=f_2 + xf_3+f_{j}, $  
     where $j$ is odd, $3\leq j\leq 2r+1$, and $0< x\leq m-3$.
      
    \item $v=yf_3+f_{j},$ 
      where $j$ is odd, $3\leq j\leq 2r+1$, and $0< y\leq m-2$.
    
    \item $v=f_2 + xf_3+f_{j}, $ 
     where $2r+1< j\leq 3r$, and $0< x\leq m-3$.

    \item $v=yf_3+f_{j},$ 
     where $2r+1< j\leq 3r$, and $0< y \leq m-2$.

\end{enumerate}
Now for (1) and (2), we compute:
\begin{align*}
    f_3 <v & \leq f_3(m-1) \\
    & = f_3 + 2a(m-2)\\
    & = [f_3 + (2a-1)(m-1)-a]+m-1-a \\
    &= f_4 + m-1-a\\
    &< f_{4}.
\end{align*}
And for (3)--(6), we compute:
\begin{align*}
    f_j <v & \leq f_j + f_3(m-2) \\
    &= f_j + 2a(m-2)\\
    &= f_j + (2a-1)(m-1)-2a+m-1 \\
    &< f_{j+1}.
\end{align*}

Hence in all cases, $v$ falls {\em between} consecutive coin values. In cases (1) and (2), we conclude that $f_3$ is a coin value in the greedy solution for $v$. This is a contradiction because on one hand $v$ was the minimum counterexample (and as such, the optimal solutions for $v$ use different coin values from the greedy solution). But then on the other hand, $f_3$ is present both in an optimal solution for $v$ and also in the greedy solution. Similarly, for cases (3), (4), (5), and (6), we conclude that $f_j$ is a coin value in both the greedy solution for $v$ and an optimal solution for $v$, which is a contradiction. Hence none of these values are minimal counterexamples. Therefore the coin system $\mathcal{F}$ is orderly.
\end{proof}  

\section{Future directions}\label{future}
In Section~\ref{extensions}, we discussed 3 families of coin systems with pattern {\footnotesize $+++-\cdots -+$}. We want to understand all such families of coin systems. Through an extensive computer search, we observed that the structure of these coin systems appears to depend on the number of coin values modulo 3. We describe our search results and subsequent conjectures in the following. 

First, we considered coin systems with $3r+1$ coins, where $r\geq 2$. Using a computer program, we generated all coin systems with 7 and 10 coin values, where the coin values are less than 100 and 50, respectively. The computer program checked whether any of these coin systems had pattern $+++---+$ or $+++------+$, respectively. None did. This confirmed the observations in~\cite{Adamaszek}. We conjecture that no coin systems with $3r+1$  coin values where $r\geq 2$ and with pattern $+++-\cdots -+$ exist. 

Next, we considered coin systems with $3r+2$ coins and with pattern {\footnotesize $+++-\cdots -+$.} Such coin systems with 5 coin values all belong to $\mathcal{D}$ in Theorem~\ref{D}, as was previously stated in Theorem~\ref{five}. We wrote a program that generated all coin systems with 8 or 11 coin values where the coin values themselves are less than 50. This search produced exactly two coin systems that did not already belong to the family $\mathcal{D}$. They were:
$$(1, 2, 4, 5, 7, 9, 12, 17)
~~~~~\text{ and }~~~~~ (1, 2, 4, 5, 7, 9, 12, 14, 17, 22, 27)$$
The similarity of these two coin systems suggests that there is one more family of coin systems which, together with $\mathcal{D}$, cover all possibilities. We call this family $\mathcal{G}$ and define it as follows. 
\begin{definition}\label{G}
    
    Let $r\geq2$. Define $\mathcal{G} = (1, g_2, \ldots, g_{3r+2})$ as follows:
$$ g_k = 
\begin{cases}
        g_{k-1}+1  & \text{if $k$ is even and } 2 \leq k  < 5\\
        g_{k-1}+2  & \text{if $k$ is odd and } 2 <  k  \leq 5\\
        g_{k-1}+2  & \text{if $k$ is even and } 5 < k  < 2r+4\\
        g_{k-1}+3  & \text{if $k$ is odd and } 5 < k < 2r+4\\
        g_{k-1}+5 & \text{if } 2r+4 \leq k\leq 3r+2
 \end{cases}
    $$
\end{definition}
It is possible that this family $\mathcal{G}$ itself belongs to an even larger family of coin systems satisfying our conditions. However, we have not observed this yet.

Finally, we considered coin systems with $3r$ coin values, where $r\geq 2$. Of course, with six value of coins, all such coin systems are members of $\mathcal{E}$ or $\mathcal{F}$ (Theorem~\ref{mainresult}). A computer search of all coin systems with 9 coin values where the values are less than 100 also yielded only coin systems from $\mathcal{E}$ and $\mathcal{F}$. However, the situation changed when we increased to 12 coin values. There are many examples of coin systems with 12 coin values that do not belong to $\mathcal{E}$ or $\mathcal{F}$. The following are the first 10 examples.

{\small
\begin{itemize}
\begin{multicols}{2}
\item $(1, 2, 4, 5, 7, 8, 10, 13, 15, 18, 21, 29)$
\item $(1, 2, 4, 5, 7, 8, 11, 13, 14, 17, 20, 26)$
\item $(1, 2, 4, 5, 7, 9, 10, 12, 15, 17, 20, 25)$
\item $(1, 2, 4, 5, 7, 9, 12, 15, 17, 20, 25, 33)$
\item $(1, 2, 4, 5, 7, 9, 12, 16, 19, 24, 31, 43)$
\item $(1, 2, 4, 5, 7, 9, 12, 16, 19, 24, 31, 46)$
\item $( 1, 2, 5, 6, 9, 10, 14, 17, 18, 22, 26, 34)$
\item $(1, 2, 6, 7, 11, 12, 17, 21, 22, 27, 32, 42)$
\item $(1, 3, 5, 8, 10, 12, 15, 17, 22, 24, 29, 36)$
\item $(1, 3, 6, 8, 11, 14, 16, 19, 24, 27, 32, 40)$
\end{multicols}
\end{itemize}
}

In light of these observations, we have the following conjecture. 

\begin{conjecture}\label{conjecture}

  \begin{enumerate}
   \item There are no coin systems with 7 or 10 coin values and pattern $+++-\cdots-+$. More generally, there are no coin systems with $3r+1$ coin values with $r\geq 2$ and pattern $+++-\cdots -+$. 
 \item A coin system with 9 coin values has pattern $+++-----+$ if and only if it is a member of $\mathcal{E}$ or $\mathcal{F}$ defined in Theorem~\ref{E} and Theorem~\ref{F}, respectively.
 
 \item A coin system with $3r+2$ coin values has pattern $+++-\cdots -+$ if and only if it is a member of $\mathcal{D}$ defined in Theorem~\ref{D} or $\mathcal{G}$ from Definition~\ref{G}.
 \end{enumerate}
\end{conjecture}

The situation with $3r$ coin values where $r\geq 4$ appears to be more complex, and hence we do not have a concrete conjecture. 
It would be interesting to determine the structure of all such coin systems. 

Recall that we used Theorem~\ref{Pearson} to prove that the coin systems $\mathcal{D}$, $\mathcal{E}$, and $\mathcal{F}$ are orderly. It would be interesting if other methods are developed to address coin systems with a large number of coin values.

 \bibliographystyle{plain} 
 \bibliography{biblio}

\end{document}